\documentclass[11pt]{amsart}

\usepackage[letterpaper]{geometry}

\usepackage{amsmath,amssymb,amsthm}
\usepackage{mathrsfs}
\usepackage{mathtools}
\usepackage{bm}
\usepackage{graphicx}
\usepackage{xcolor}
\usepackage{tikz-cd}
\usetikzlibrary{matrix}

\usepackage[pagebackref=true]{hyperref}
\hypersetup{
    colorlinks=true,
    allcolors=blue
}

\newtheorem{theorem}{Theorem}[section]
\newtheorem{lemma}[theorem]{Lemma}
\newtheorem{proposition}[theorem]{Proposition}

\newtheorem{corollary}[theorem]{Corollary}

\theoremstyle{definition}
\newtheorem{definition}[theorem]{Definition}
\newtheorem{example}[theorem]{Example}

\theoremstyle{remark}
\newtheorem{remark}[theorem]{Remark}

\definecolor{oro}{rgb}{0.80,0.58,0.05}
\definecolor{darkgreen}{rgb}{0.0,0.4,0.0}
\definecolor{lav}{rgb}{0.80,0.0,0.95}

\newcommand{\Z}{\mathbb{Z}}
\newcommand{\C}{\mathbb{C}}
\newcommand{\R}{\mathbb{R}}
\newcommand{\T}{\mathbb{T}}
\newcommand{\Lam}{\pmb{\Lambda}}

\DeclareMathOperator{\rank}{rank}

\newcommand{\lamkn}[2]{\lambda^{#2}_{#1}}
\newcommand{\convexhull}[1]{ \operatorname{Conv}(#1)}
\newcommand{\leaf}[1]{\mathcal{L}_{#1}}
\newcommand{\strat}[1]{E_{#1}}

\newcommand{\stratSS}[1]{E_{#1}^*}
\newcommand{\stratS}[1]{\stratSS{#1}}

\title[Holomorphic Linear $\C^k$-Actions]{Holomorphic Linear $\C^k$-Actions, Trace Foliations, and Higher-Rank Poincaré Dynamics}
\date{}
\author{Aubin Arroyo, Carlos Cabrera, José Seade, and Alberto Verjovsky}

\address{Instituto de Matem\'aticas, Universidad Nacional Aut\'onoma de M\'exico\\
Apartado Postal 273, Administraci\'on de Correos No.~3\\
C.P. 62251 Cuernavaca, Morelos, M\'exico}

\email{[aubinarroyo@im.unam.mx]}
\email{[carloscabrerao@im.unam.mx]}
\email{[jseade@im.unam.mx]}
\email{[alberto@matcuer.unam.mx]}

\begin{document}

\begin{abstract}
	We study the orbit decomposition on $\C^n$ generated by diagonal holomorphic $\C^k$-actions in the higher-rank setting of the classical Poincaré--Siegel dichotomy for linear vector fields.
	The coordinate stratification determines the dimensions and isotropy groups of the leaves and, under a maximal-rank condition, gives a precise description of the orbit structure on every coordinate stratum.

	For configurations in the Poincaré domain, a separating real direction provides a global conical model of the punctured orbit foliation by its traces on Euclidean spheres.
	We construct the corresponding radially reparametrized action on a sphere, describe its leaves as homogeneous spaces, and prove that orbit closures are constrained by coordinate supports.
	In particular, a limit point cannot acquire a new nonzero coordinate, although nonclosed trace leaves may also accumulate within a fixed support stratum.
	We show that the geometry of the weight configuration determines the complex dimensions of the leaves, whereas the arithmetic of their effective isotropy groups determines their diffeomorphism types.
	This gives rise to a threshold phenomenon across the coordinate stratification: the diffeomorphism type of the leaves is rigid in the low- and high-dimensional regimes, but becomes arithmetically unstable in the intermediate range $k<|I|<2k$.
	Finally, we show that these singular foliations admit canonical local transverse holomorphic structures in the spirit of Haefliger's transverse geometry for regular foliations.
	These structures determine intrinsic transverse pseudogroups, yielding a well-defined local transverse holomorphic geometry for the orbit foliation.
\end{abstract}


\maketitle

\section{Introduction}

Holomorphic flows generated by linear vector fields provide some of the most elementary models in local complex dynamics.
Their higher-rank counterparts, given by diagonal holomorphic actions of $\C^k$, lead to singular orbit foliations whose geometry is controlled by the associated weight configuration.

The Poincaré--Siegel dichotomy, originating in the work of Poincaré and Dulac and developed further by Arnold \cite{Arnold}, Guckenheimer \cite{Gu}, Camacho, Kuiper, and Palis \cite{CKP}, among others, remains one of the guiding principles in the subject.
In higher rank, the two regimes give rise to different geometric phenomena.
The Poincaré regime, in particular, provides a natural framework for studying the topology, stratification, and transverse geometry of orbit foliations.

In this article, we study diagonal holomorphic actions of the additive complex Lie group $\C^k$ on $\C^n$.
These actions are generated by commuting families of diagonal linear holomorphic vector fields and are determined by a weight configuration
\[
	\Lam=\{\Lambda_1,\dots,\Lambda_n\}\subset\C^k.
\]
They are given by
\[
	(T,z)\longmapsto
	\bigl(
	z_1e^{\langle\Lambda_1,T\rangle},
	\dots,
	z_ne^{\langle\Lambda_n,T\rangle}
	\bigr),
	\qquad T\in\C^k.
\]

The associated orbit decomposition is a Stefan--Sussmann singular foliation in the smooth sense.
Its infinitesimal generators span an involutive coherent holomorphic subsheaf of the tangent sheaf, and on every coordinate stratum the resulting distribution has constant rank and defines a regular holomorphic foliation.
Its leaf dimensions and isotropy groups may vary from point to point but are constant on each stratum of the coordinate-plane stratification.
This setting simultaneously extends the classical theory of rank-one linear flows and the study of diagonal complex actions arising in the theory of L\'opez de Medrano--Verjovsky--Meersseman manifolds \cite{LdM-V, Meersseman, Verjovsky, Santiago2}.
In particular, configurations in the Siegel domain were introduced by Meersseman in the study of compact complex manifolds constructed from diagonal actions.
We focus on configurations in the Poincaré domain and study the resulting orbit foliations through their global geometry, the topology and asymptotic behavior of their leaves, and their canonical transverse holomorphic structures.

The first part of the paper is devoted to the geometry along the leaves of the orbit foliation, including their dimensions, topology, asymptotic behavior, and diffeomorphism types.
The final part is devoted to the transverse holomorphic geometry of the orbit foliation.
We construct Stefan--Sussmann transversals, local product models, and intrinsic transverse pseudogroups, showing that the resulting transverse holomorphic geometry is independent of all auxiliary choices.

A central theme of the paper is the interaction between the geometry and the arithmetic of the weight configuration.
The geometry of the weight configuration determines the complex dimensions of the leaves, whereas the arithmetic of the corresponding effective isotropy groups determines their diffeomorphism types.

We first describe the global organization of the orbit foliation through the coordinate stratification of $\C^n$.
The pattern of vanishing coordinates determines both the dimension of an orbit and its isotropy subgroup, and the coordinate strata provide a natural decomposition of the ambient space into regions governed by subconfigurations of the weights.
Under a maximal-rank genericity condition, strata of complex dimension at most $k$ consist of a single leaf, whereas the higher-dimensional strata are foliated by leaves of maximal complex dimension.

The geometry of the orbit foliation is largely governed by the position of the origin relative to the convex hull of the weight configuration $\Lam$.
If
\[
	0\notin\operatorname{Conv}(\Lam),
\]
we say that $\Lam$ lies in the \emph{Poincaré domain}; otherwise, when $0\in\operatorname{Conv}(\Lam)$, it lies in the \emph{Siegel domain}.

The two regimes lead to markedly different orbit geometries.
In the Siegel case, suitable open subsets may admit compact quotients by the diagonal action, a phenomenon underlying the constructions associated with L\'opez de Medrano, Verjovsky, and Meersseman.
In the Poincaré case, by contrast, the origin is dynamically isolated, and the punctured orbit foliation admits a conical description.

More precisely, every nontrivial orbit meets each Euclidean sphere transversely.
Its intersections with a sphere define a singular foliation, which we call the \emph{trace foliation}.
This spherical foliation encodes the topology and asymptotic behavior of the punctured ambient dynamics.

One of the main ingredients, Lemma~\ref{thm:trace-main}, gives a homogeneous description of the trace foliation.
For every admissible direction $\xi$, a codimension-one real linear subgroup $H_\xi\subset\C^k$ acts on the sphere, and its orbits are precisely the trace leaves.
Thus, each trace leaf is an immersed homogeneous space of the form $H_\xi/\Gamma_z$, where $\Gamma_z$ is its stabilizer subgroup.
Moreover, the punctured ambient foliation is homeomorphic to the open cone over the trace foliation.
When $k=1$, this recovers the classical conical picture described by Guckenheimer \cite{Gu} for linear holomorphic flows.

In Lemma~\ref{thm:omega-strata}, we analyze orbit closures of trace leaves.
Diagonality imposes a universal support-monotonicity constraint: a limit point cannot acquire a new nonzero coordinate.
This does not exclude same-support recurrence.
Nonclosed trace leaves may accumulate on other trace leaves with the same support, for example through quasiperiodic winding in compact factors.

We also study how the orbit foliation changes under perturbations of the weight configuration.
The leaves are quotients of complex vector groups by their effective isotropy groups and are therefore classified, up to diffeomorphism, by the rank of the latter.
Thus, while the geometry of the weight configuration determines the complex dimensions of the leaves, the arithmetic of their effective isotropy groups determines their diffeomorphism types.
This gives rise to a threshold phenomenon across the coordinate stratification: the diffeomorphism type of the leaves is rigid in the low- and high-dimensional regimes, but becomes arithmetically unstable in the intermediate range $k<|I|<2k$.

The orbit decomposition also admits an intrinsic local transverse holomorphic structure.
It is described by Stefan--Sussmann transversals, local product models, and transverse pseudogroups, all uniquely determined up to germ biholomorphism.

In the Poincaré regime, the same construction applies to the (real) trace foliation on the sphere, endowing it with a natural canonical transverse holomorphic structure.

Thus, the classical theory of rank-one linear holomorphic flows appears as part of a broader framework of higher-rank complex actions in which the coordinate stratification governs the geometry, topology, dynamics, and canonical holomorphic transverse structures of the orbit and trace foliations.

The paper is organized as follows.
Section~\ref{sec:preliminaries} introduces diagonal $\C^k$-actions, isotropy groups, and the coordinate stratification of the associated orbit foliation.
Section~\ref{sec:PoincareDynamics} studies the Poincaré domain and develops the geometry of trace foliations on spheres.
Section~\ref{sec:stability} studies the topology of the orbit foliation, establishes a threshold phenomenon for the diffeomorphism types of the leaves, and derives the failure of stratified topological stability.
Section~\ref{sec:transverse} develops the natural stratified holomorphic transverse geometry of the orbit and trace foliations.

\section{Diagonal \texorpdfstring{$\C^k$}{C\string^k}-Actions and Leaf Geometry}\label{sec:preliminaries}

The class of holomorphic actions of $\C^k$ on $\C^n$ considered throughout the paper is determined by vectors $\Lambda_j=(\lamkn{1}{j},\dots,\lamkn{k}{j})\in\C^k$, for $j\in\{1,\ldots,n\}$.
We call $\Lam=\{\Lambda_1,\dots,\Lambda_n\}$ the weight configuration of the action.

The configuration $\Lam$ determines a diagonal holomorphic action of the additive complex Lie group $\C^k$ on $\C^n$, defined by
\[
	\Phi:\C^k\times\C^n\longrightarrow\C^n, \qquad \Phi(T,z)= \bigl( z_1e^{\langle\Lambda_1,T\rangle}, \dots, z_ne^{\langle\Lambda_n,T\rangle} \bigr),
\]
where $T=(t_1,\dots,t_k)\in\C^k$ and
\[
	\langle\Lambda_j,T\rangle = \sum_{\nu=1}^k \lamkn{\nu}{j} \, t_\nu .
\]

The corresponding infinitesimal generators are the $k$ commuting holomorphic vector fields
\begin{equation} \label{eq:infinitesimalGenerators}
	X_\nu=\sum_{j=1}^n \lamkn{\nu}{j} \, z_j \, \frac{\partial}{\partial z_j}, \qquad \nu=1,\dots,k.
\end{equation}

The orbits of this action define a Stefan--Sussmann singular foliation $\mathcal F(\Lam)$ in the smooth sense.
The holomorphic vector fields in~\eqref{eq:infinitesimalGenerators} span an involutive coherent subsheaf of $T_{\C^n}$; on every coordinate stratum its rank is constant and it defines a regular holomorphic foliation.
Writing $\convexhull{\Lam}$ for the convex hull of $\Lam$, the convex geometry of the weight configuration distinguishes two basic regimes, according to whether $0\in\convexhull{\Lam}$ or not.

\begin{definition}
	The configuration $\Lam$ is said to lie in the \emph{Poincaré domain} if $0 \notin \convexhull{\Lam}$, and in the \emph{Siegel domain} if $0 \in \convexhull{\Lam}$.
	The associated action is called a \emph{Poincaré action} or a \emph{Siegel action}, respectively.
\end{definition}

These notions extend the classical rank-one dichotomy.
In the Poincaré domain, a real affine hyperplane strictly separates the origin from the weights.
Equivalently, there is a real direction with strictly positive real pairing against every weight; this produces radial transversality and the conical behavior studied below.

\medskip
From now on we assume that $\Lam$ lies in the Poincaré domain and
satisfies the following genericity condition: every subconfiguration
$\{\Lambda_j:j\in I\}$ has maximal possible rank,
\[
	\rank_{\C}\{\Lambda_j:j\in I\}=\min\{|I|,k\},
	\qquad I\subset\{1,\dots,n\}.
\]
A configuration satisfying this condition will be called generic.

The origin is always a singular point of the foliation.
More generally, the dimension of the orbit through a point is governed by the weights corresponding to its nonzero coordinates.

To describe the geometry of the orbit foliation, we associate to each point the set of its nonzero coordinates.
For $z=(z_1,\dots,z_n)\in\C^n$, let
\[
	I(z)=\{j\in\{1,\dots,n\}: z_j\neq 0\}
\]
denote the support of $z$.
Thus $I(z)$ specifies the coordinates that are active along the orbit of $z$.
Thus, the orbit through $z$ is singular precisely when the weights $\{\Lambda_j:j\in I(z)\}$ fail to span $\C^k$.

The notion of support induces a natural coordinate stratification of $\C^n$, which we use throughout the paper.
For each subset $I\subset\{1,\dots,n\}$, set
\[
	\strat{I}=\{w\in\C^n:w_j=0\text{ for }j\notin I\}, \qquad \stratSS{I}=\{w\in \strat{I}:w_j\neq0\text{ for }j\in I\}.
\]
Then
$\stratSS{I}$
is the set of points with support exactly $I$, and
\begin{equation}\label{eq:stratification}
	\C^n=\bigsqcup_{I\subset\{1,\dots,n\}}\stratSS{I}.
\end{equation}
%For convenience we write $\stratS{I}:=\stratSS{I}$. {\color{oro} Usamos ambas notaciones un montón. Creo que hay que usar solo una}

Each stratum $\stratS{I}$ is invariant under the action: zero coordinates remain zero along every orbit, while nonzero coordinates remain nonzero.
Hence, the foliation $\mathcal F(\Lam)$ is naturally stratified by the coordinate decomposition.

Let $z\in\C^n$.
The isotropy subgroup of $z$ is
\[
	G_z=\{T\in\C^k:\Phi(T,z)=z\}.
\]

The coordinate stratification gives a simple property of the isotropy groups: In terms of the support $I(z)$ the isotropy is given by
\[
	G_z= \left\{ T\in\C^k: e^{\langle\Lambda_j,T\rangle}=1 \text{ for all }j\in I(z) \right\}.
\]

The leaf $\leaf{z}$ through $z$ is naturally identified with the homogeneous space
\[
	\leaf{z}\simeq \C^k \big/ G_z.
\]
Since $\C^k$ is Abelian, every subgroup is normal.
Consequently each leaf is canonically a complex Abelian quotient Lie group.

Let $I=I(z)$ and set $m=|I|$.
Then the orbit through $z$ is contained in the coordinate stratum $\stratSS{I}$, which is biholomorphic to $(\C^*)^m$.
%
%For any subset $J\subset{1,\ldots,n}$, define
%\[
%	q(J):=\rank_{\C}\{\Lambda_j:j \in J\},
%\]
%that is, the dimension of the complex vector subspace of $\C^k$ spanned by the weights indexed by $J$. For $z \in \C^n$, we set $q(z) := q(I(z))$. Thus, $q(z)$ is the rank of the family of active weights at $z$.
%
Also, let
\[
	q(z):=\rank_{\C}\{\Lambda_j:j \in I(z)\},
\]
that is, the dimension of the complex vector subspace of $\C^k$ generated by the active weights.

\begin{proposition}\label{prop:leaf-dimension}
	Let $z\in\C^n$ and let $I(z)$ be its support.
	Then the leaf $\leaf{z}$ through $z$ has complex dimension
	\[
		\dim_\C(\leaf{z})= q(z).
	\]
\end{proposition}

\begin{proof}
	The differential of the action at $z$,
	\[
		d\Phi_z:\C^k\longrightarrow T_z\C^n,
	\]
	is generated by the tangent vectors
	\[
		\left(\lamkn{\nu}{1}z_1,\ldots,\lamkn{\nu}{n}z_n\right), \qquad \nu=1,\dots,k.
	\]
	Only the coordinates with $j\in I(z)$ contribute to the rank.
	Since $z_j\neq0$ for every $j\in I(z)$, multiplying the $j$-th coordinate by the nonzero scalar $z_j$ does not change the rank.
	Therefore, the tangent space to the orbit at $z$ has complex dimension equal to
	\[
		\rank_{\C}\{\Lambda_j:j\in I(z)\}.
	\]
	Since the complex dimension of the orbit equals the dimension of its
	tangent space, the result follows.
\end{proof}

It follows that dimension drops occur precisely along those strata $\stratSS{I}$ for which the active weight configuration $\{\Lambda_j:j\in I\}$ has rank strictly smaller than $k$.

On the principal stratum $(\C^*)^n$, all coordinates are nonzero.
Hence, the isotropy subgroup is constant on this stratum.
We denote it by $G_{\mathrm{pr}}$ and call it the \emph{principal isotropy subgroup}.
The leaves contained in this Zariski open stratum are called the \emph{principal leaves}.

Denote by $\T ^r = \mathbb{R}^r / \Z^r$ the real torus of dimension $r$.
The following proposition gives a general description of the topology of principal leaves.
Under the genericity condition, $\mathrm{rank}_{\C}\{\Lambda_1,\dots,\Lambda_n\}=k$, so the kernel $K_{\{1,\dots,n\}}=\ker A_{\{1,\dots,n\}}=\{0\}$; by Lemma~\ref{lem:isotropy-splitting} below, this implies that $G_{\mathrm{pr}}$ is discrete.
\begin{proposition} \label{prop:principal-topology}
	Assume the genericity condition.
	If the principal isotropy subgroup $G_{\mathrm{pr}}\subset\C^k$ has rank $r$ (as a free abelian group), then every principal leaf is diffeomorphic to
	\[
		\R^{2k-r}\times\T^r.
	\]
\end{proposition}

\begin{proof}
	The principal leaf is naturally identified with $\C^k/G_{\mathrm{pr}}$.
	As a real Lie group, $\C^k\simeq\R^{2k}$.
	Since $G_{\mathrm{pr}}$ is a discrete subgroup of rank $r$, it can be identified, after a real linear change of coordinates, with $\Z^r\subset\R^{2k}$.
	Hence
	\[
		\C^k/G_{\mathrm{pr}} \simeq \R^{2k}/\Z^r \simeq \R^{2k-r}\times\T^r.
	\]
\end{proof}

We now turn to the topology of the leaves contained in the lower coordinate strata.

\begin{theorem}\label{thm:coordinate-strata}
	Assume that $\Lam$ is a generic configuration in the Poincaré domain.
	Let $I\subset\{1,\dots,n\}$ with $|I|=m$, and consider the coordinate stratum $\stratSS{I}\subset\C^n$.
	Then:
	\begin{enumerate}
		\item if $m>k$, then $\stratSS{I}$ is foliated by mutually diffeomorphic
		      leaves of complex dimension $k$;
		\item if $m\le k$, then $\stratSS{I}$ consists of a single leaf.
	\end{enumerate}
\end{theorem}

\begin{proof}
	Every point of $\stratSS{I}$ has support exactly $I$.
	For $z\in\stratSS{I}$, the isotropy subgroup is
	\[
		G_z= \left\{ T\in\C^k: \langle\Lambda_j,T\rangle\in 2\pi i\Z \text{ for all } j\in I \right\}.
	\]
	Indeed, the coordinates with $j\notin I$ vanish and impose no condition,
	while $z_j\neq0$ for $j\in I$.
	Hence, $G_z$ depends only on the support $I$, and is therefore constant throughout the stratum $\stratSS{I}$.

	By Proposition~\ref{prop:leaf-dimension}, the complex dimension of the leaf through any point of $\stratSS{I}$ is
	\[
		\rank_{\C}\{\Lambda_j:j\in I\}.
	\]
	By the maximal-rank hypothesis, this dimension is equal to
	$\min\{m,k\}$.

	If $m>k$, then $\min\{m,k\}=k$.
	Hence $\stratSS{I}$ is foliated by leaves of maximal complex dimension $k$.
	Since the isotropy subgroup is constant on the whole stratum, these leaves are mutually diffeomorphic.

	If $m\le k$, then the leaf dimension is $m=\dim_\C \stratSS{I}$.
	Thus, every leaf in $\stratSS{I}$ is open in the stratum.
	Since $\stratSS{I}\simeq(\C^*)^m$ is connected, the stratum consists of a single leaf.
\end{proof}

Under these hypotheses, the orbit structure on a coordinate stratum $\stratSS{I}$ depends only on the cardinality $|I|$ of the support.

From the complex-analytic viewpoint, the principal leaves are quotients $\C^k/G_{\mathrm{pr}}$.
Thus, they are complex Abelian Lie groups, whose analytic type may range from Stein examples to Cousin-type quotients.

The preceding discussion shows that the orbit foliation is governed by three complementary aspects of the weight configuration: the ranks of its subconfigurations determine the dimensions of the leaves, the convex position of the weights distinguishes the Poincaré and Siegel regimes, and the arithmetic of the isotropy groups determines the topology of the leaves.

\begin{example}\label{ex:C3C4}
	Consider the diagonal action of $\C^3$ on $\C^4$ determined by the weights
	\[
		\Lambda_1=(1,0,0),\qquad \Lambda_2=(0,1,0),\qquad \Lambda_3=(0,0,1),\qquad \Lambda_4=(1,1,1).
	\]
	This configuration lies in the Poincaré domain, since all weights belong
	to the open half-space $x_1+x_2+x_3>0$.
	Moreover, every subconfiguration has maximal rank.

	By Proposition~\ref{prop:leaf-dimension}, the dimension of the leaf through a point $z\in\C^4$ is the rank of the active weights $\{\Lambda_j:j\in I(z)\}$.
	Hence, leaves with support of cardinality $1$, $2$, and $3$ have dimensions $1$, $2$, and $3$, respectively, whereas points with full support also have leaves of dimension $3$.

	Thus, the foliation has leaves of all possible dimensions \(1,2,3\).
	On the principal stratum \((\C^*)^4\), the first three weights form the standard basis of \(\C^3\).
	Hence, if \(T=(t_1,t_2,t_3)\) belongs to the principal isotropy subgroup, then $\{ t_1,t_2,t_3 \} \subset 2\pi i\Z$.
	The fourth condition, $e^{t_1+t_2+t_3}=1$, is then automatically satisfied.
	Therefore $G_{\mathrm{pr}}=(2\pi i\Z)^3$.
	It follows that the corresponding principal leaves are diffeomorphic to $\R^3\times\T^3$.
	This illustrates how the coordinate stratification produces lower-dimensional leaves, while the principal leaves retain the topology of the quotient $\C^3/(2\pi i\Z)^3$.
\end{example}

The next section focuses on the Poincaré case, where the foliation admits a conical model through the trace foliation induced on spheres.

\section{Poincaré Dynamics and Trace Foliations on Spheres}
\label{sec:PoincareDynamics}

In this section we study the punctured orbit foliation associated with a Poincaré configuration.
For $r>0$, put
\[
	S_r=\{z\in\C^n:\|z\|=r\}.
\]
The intersections with $S_r$ of the nonzero ambient leaves form the
\emph{trace foliation}, denoted by $\widehat{\mathcal F}_r(\Lam)$.

\begin{proposition}[Conical model]
	\label{prop:cone}
	Assume that $\Lam$ lies in the Poincaré domain.
	Fix
	\[
		C_{\Lam}=\{\xi\in\C^k_{\R}:\operatorname{Re}\langle\Lambda_j,\xi\rangle>0, \ j=1,\dots,n\},
	\]
	and choose $\xi\in C_{\Lam}$.
	Then
	\[
		\Theta_r:S_r\times\R\longrightarrow\C^n\setminus\{0\}, \qquad (z,t)\longmapsto\Phi(t\xi,z),
	\]
	is a homeomorphism.
	It maps the product of every trace leaf with $\R$ onto the corresponding nonzero ambient leaf.
	Thus the punctured orbit foliation is homeomorphic to the open cone over the trace foliation.
\end{proposition}

\begin{proof}
	The separating-hyperplane theorem gives $C_{\Lam}\ne\varnothing$.
	Figure~\ref{F1} is a real two-dimensional illustration of this condition.
	For $z\ne0$,
	\[
		\|\Phi(t\xi,z)\|^2= \sum_{j=1}^n|z_j|^2e^{2t\operatorname{Re}\langle\Lambda_j,\xi\rangle}
	\]
	is strictly increasing in $t$, tends to $0$ as $t\to-\infty$, and tends to $\infty$ as $t\to\infty$.
	Hence each $\R\xi$-orbit meets $S_r$ at exactly one point.
	The unique hitting time depends continuously on the initial point by strict monotonicity (equivalently, by the implicit function theorem), which gives the stated homeomorphism and its inverse.
\end{proof}

\begin{figure}[ht]
	\centering
	\includegraphics[width=0.55\textwidth]{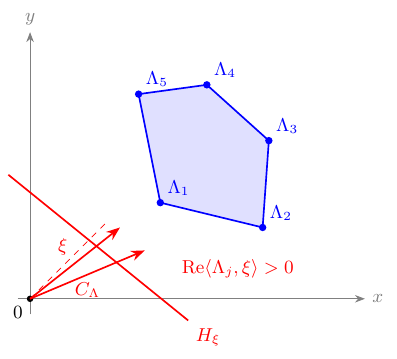}
	\caption{A real two-dimensional illustration of the Poincaré condition.
		The convex hull of the weights is strictly separated from the origin.
		Equivalently, there is a real direction $\xi$ such that $\operatorname{Re}\langle\Lambda_j,\xi\rangle>0$ for every $j$.
	}
	\label{F1}
\end{figure}

Fix $\xi\in C_{\Lam}$ and a real hyperplane $H_\xi\subset\C^k_{\R}$ complementary to $\R\xi$, so that
\[
	\C^k_{\R}=H_\xi\oplus\R\xi.
\]
For $U\in H_\xi$ and $z\in S_r$, let $\tau(U,z)$ be the unique real
number for which
\[
	\Phi(U+\tau(U,z)\xi,z)\in S_r.
\]
Such $\tau$ exists and depends smoothly on $(U,z)$: the function $s\mapsto\|\Phi(U+s\xi,z)\|^2$ has strictly positive derivative 
$$
	2\sum_j|z_j|^2e^{2s\,\mathrm{Re}\langle\Lambda_j,\xi\rangle} \mathrm{Re}\langle\Lambda_j,\xi\rangle >0, 
$$
at every $z\ne0$, so the implicit function theorem applies to the equation $\|\cdot\|^2=r^2$.
Set
\[
	\Psi_U(z)=\Phi(U+\tau(U,z)\xi,z).
\]

\begin{lemma}[Trace action and homogeneous leaves]
	\label{thm:trace-main}
	For every $r>0$ the maps $\Psi_U$ define a smooth action of $H_\xi$ on $S_r$.
	Its orbits are precisely the trace leaves.
	If $z\in S_r$ has support $I(z)$ and
	\[
		q(z)=\rank_{\C}\{\Lambda_j:j\in I(z)\},
	\]
	then the trace leaf through $z$ has real dimension $2q(z)-1$ and is an immersed homogeneous space
	\[
		\widehat O_z\cong H_\xi/\Gamma_z, \qquad \Gamma_z=\{U\in H_\xi:\Psi_U(z)=z\}.
	\]
\end{lemma}

\begin{proof}
	The uniqueness of the radial correction and the action identity for $\Phi$ imply
	\[
		\Psi_U\circ\Psi_V=\Psi_{U+V}.
	\]
	Every ambient orbit is the union of the $\R\xi$-orbits through its
	intersection with $S_r$, so the $H_\xi$-orbits are exactly the trace
	leaves.
	The ambient orbit through $z$ has complex dimension $q(z)$, and the $\R\xi$ direction is everywhere transverse to the sphere; therefore its intersection with $S_r$ has real dimension $2q(z)-1$.
	The homogeneous-space description is the standard orbit--stabilizer identification for a Lie-group action.
\end{proof}

\begin{remark}
	The action $\Psi$ is a \emph{radially reparametrized} $H_\xi$-action; it is not in general the literal restriction of the original linear action to $S_r$.
	Different choices of $\xi$ and of a complementary hyperplane produce orbit-equivalent descriptions of the same trace foliation.
\end{remark}

\begin{remark}
	Since $H_\xi$ is a real vector group and $\Gamma_z$ is closed, every trace leaf is a connected abelian Lie group.
	Hence it is diffeomorphic to $\R^{2q(z)-1-s}\times\T^s$ for some $s\ge0$.
\end{remark}

We next record the universal restriction on orbit closures supplied by the diagonal action.

\begin{lemma}[Support monotonicity for trace-orbit closures]
	\label{thm:omega-strata}
	Let $z\in S_r$ and define
	\[
		\omega(z)=\left\{p\in S_r:\exists U_\nu\in H_\xi,\ \|U_\nu\|\to\infty,\ \Psi_{U_\nu}(z)\to p\right\}.
	\]
	Then every $p\in\omega(z)$ satisfies
	\[
		I(p)\subseteq I(z).
	\]
	Consequently,
	\[
		\omega(z)\subseteq
		\bigcup_{J\subseteq I(z)}\bigl(S_r\cap\stratSS{J}\bigr).
	\]
\end{lemma}

\begin{proof}
	If $z_j=0$, then the $j$th coordinate of $\Phi(T,z)$ is zero for every $T\in\C^k$, and the radial correction does not alter this fact.
	Hence every coordinate vanishing at $z$ vanishes on the trace orbit and on its closure.
\end{proof}

\begin{remark}[Same-support recurrence]
	\label{rem:same-support}
	The inclusion in Lemma~\ref{thm:omega-strata} need not be strict for points outside the original trace leaf.
	Nonclosed trace leaves may accumulate on other leaves with the same support.
	For example, when $k=1$, $n=2$, and $\Lambda_1=1$, $\Lambda_2=\sqrt2$, the action of $H_\xi=i\R$ on a sphere is
	\[
		s\longmapsto(e^{is}z_1,e^{i\sqrt2s}z_2).
	\]
	For $z_1z_2\ne0$, its orbit is dense in the torus determined by $|z_1|$ and $|z_2|$.
	Thus orbit closure has a same-support component.
	Coordinate degeneration is therefore only one source of accumulation.
\end{remark}

In particular, no uniform rank-one statement describes all $\omega$-limit sets without supplementary assumptions on the weights.
Depending on their arithmetic and on their real and imaginary parts, trace dynamics may be periodic, quasiperiodic, or may also exhibit accumulation on lower-support strata.

\medskip

We now summarize the previous discussion.
Recall that if  $z\in S_r$, the number $q(z)$ is the complex dimension of the vector subspace of $\C^k$ generated by the active weights in the support $I(z)$.
%
%For every subset $ I\subset\{1,\ldots,n\}$, let 
%\[
%	q(I):=\operatorname{rank}_{\C}\{\Lambda_j:j\in I\},
%\] 
%denote the complex dimension of the vector subspace of $\C^k$ generated by the weights $\{\Lambda_j:j\in I\}$.
%If $z\in S_r$, we write $ q(z):=q(I(z))$.
Under the genericity assumption one has
$$
	q(z)=\min\{|I(z)|,k\}.
$$
The preceding results combine to give the following global description of the trace foliation.
\begin{theorem}[Global structure of the trace foliation]
	Assume that $\Lambda$ lies in the Poincaré domain with generic weights, and fix $r>0$.
	Let $\widehat{\mathcal F}_r(\Lambda)$ be the trace foliation on $S_r$.
	Then:

	\begin{enumerate}
		\item For a suitable real $(2k-1)$-dimensional vector group
		      $
			      H_\xi\subset \C^k_{\R},
		      $
		      the trace leaves are precisely the $H_\xi$-orbits.
		      Hence each trace leaf $\widehat{\leaf{z}}$ is an immersed homogeneous manifold
		      $
			      \widehat{\leaf{z}}\cong H_\xi/\Gamma_z,
		      $
		      where $\Gamma_z$ is its isotropy subgroup.

		\item For each $z\in S_r$, the trace leaf $\widehat{\leaf{z}}$ has real dimension $2q(z)-1$.
%		      $$
%			      \dim_{\R}\widehat {\leaf{z}} =2q(z)-1 = 2\min\{| I(z) |,k\} - 1.
%		      $$

%		\item The closure of every trace leaf is saturated by trace leaves and respects the support stratification.
%		      More precisely, if
%		      \[
%			     \widehat{\leaf{w}} \subseteq  \operatorname{closure}(\widehat{\leaf{z}}),
%		      \]
%		      then
%		      \[
%			      I(w)\subseteq I(z).
%		      \]

\item The closure of each trace leaf is saturated by trace leaves and is compatible with the support stratification: if $\widehat{\leaf{w}}\subseteq \operatorname{cl}\bigl(\widehat{\leaf{z}}\bigr)$, then $I(w)\subseteq I(z)$.

	\end{enumerate}
\end{theorem}

\begin{proof}
	Statements (1) and (2) follow from Lemma~\ref{thm:trace-main} and the genericity assumption.
	Since the closure of an $H_\xi$-orbit is $H_\xi$-invariant, it is saturated by trace leaves.
	The support inclusion in (3) then follows from Lemma~\ref{thm:omega-strata}.
\end{proof}

\begin{remark}
	Orbit closures may exhibit two distinct types of accumulation: recurrence within a fixed support stratum and degeneration to lower-support strata.
	The former is illustrated in Remark \ref{rem:same-support}.
\end{remark}

\section{Topology and Stability}\label{sec:stability}

In the previous section we showed that the orbit decomposition is stratified by the coordinate supports, and the complex dimension of the leaves on each stratum is determined by the rank of the corresponding subset of weights.

We now investigate the topology of these leaves.
Since each leaf is an orbit of the action of $\C^k$ which is an Abelian Lie group, every leaf is a connected complex Abelian Lie group.
More precisely, if $\leaf{}$ is a leaf of complex dimension $q$, then $\leaf{}\cong\C^q/\Gamma$, where $\Gamma\subset\C^q$ is a discrete subgroup.

Every discrete subgroup of $\C^q$ is a free Abelian group of finite rank.
Hence, there exists a unique integer $r=\operatorname{rank}(\Gamma)$ such that $ \Gamma\simeq\Z^r$.
Consequently,
$$
	\leaf{}\simeq\R^{2q-r}\times\T^r.
$$
Thus, while the geometry of the weight configuration determines the complex dimensions of the leaves, the arithmetic of the corresponding isotropy subgroups determines their diffeomorphism types.

\medskip
As shown in Section~\ref{sec:PoincareDynamics}, the orbit stratification contains strata whose leaves have every complex dimension between $1$ and $k$.
The purpose of this section is to understand how the topology of these leaves varies across the orbit stratification.
We shall see that the stratification naturally decomposes into three distinct regimes.
In the low-dimensional regime the diffeomorphism type of the leaves is rigid; in the high-dimensional regime it is generically rigid (for an open dense set of weight configurations); whereas in the intermediate regime it depends on arithmetic properties of the isotropy subgroup.
This threshold phenomenon will ultimately lead to the failure of stratified stability.

\medskip

For a coordinate stratum $\stratS{I}$, consider the linear map 
$$ 
	A_I:\C^k\longrightarrow\C^{|I|}, \qquad A_I(T)= (\langle\Lambda_j,T\rangle)_{j\in I}.
$$
Set
$$
	K_I=\ker(A_I) =	\bigcap_{j\in I}\ker\Lambda_j.
$$

With this notation, we have the following lemma.

\begin{lemma}\label{lem:isotropy-splitting}
	Let $z\in\C^n$ and let $I=I(z)$ be its support.
	Let $ K_I $ be as above and choose a complex vector subspace $F_I\subset\C^k$ such that $\C^k=K_I\oplus F_I.$
	%(We write $F_I$ rather than $\strat{I}$ to avoid collision with the coordinate subspace $\strat{I}\subset\C^n$ of Section~\ref{sec:preliminaries}.)
	Then the isotropy subgroup of the leaf through $z$ decomposes as $G_z = K_I\oplus G_z^{\mathrm{eff}}$, where $G_z^{\mathrm{eff}} = G_z\cap F_I$.
	Moreover, $G_z^{\mathrm{eff}}$ is a discrete subgroup of $F_I$, and the quotient group $\C^k/G_z$ is naturally isomorphic to $F_I/G_z^{\mathrm{eff}}$.
\end{lemma}

Thus the connected summand $K_I$ is determined purely by the linear geometry of the weights, while the discrete summand $G_z^{\mathrm{eff}}$ carries the arithmetic information relevant to the diffeomorphism type of the leaf.

\begin{proof}
	Since $z\in \stratS{I}$, the isotropy equations are 
	$$ 
		e^{\langle\Lambda_j,T\rangle}=1, \qquad j\in I.
	$$
	Equivalently,
	$$
		A_I(T)\in (2\pi i\Z)^{|I|}.
	$$
	In particular, every element of $K_I=\ker A_I$ fixes $z$, so $K_I\subset G_z$.

	Let $T\in G_z$.
	Write uniquely $T=T_0+T_1$ with $T_0\in K_I$ and $T_1\in F_I$.
	Since $T_0\in K_I\subset G_z$ and $T\in G_z$, it follows that $T_1=T-T_0\in G_z$.
	Hence $T_1\in G_z\cap F_I=G_z^{\mathrm{eff}}$.
	This proves 
	$$ 
		G_z=K_I\oplus G_z^{\mathrm{eff}}.
	$$

	The restriction $A_I|_{F_I}:F_I\to A_I(\C^k)$ is a real-linear isomorphism (since $F_I\cap K_I=\{0\}$ by the direct-sum decomposition).
	Therefore 
	$$ 
		G_z^{\mathrm{eff}} = (A_I|_{F_I})^{-1}\big((2\pi i\Z)^{|I|} \cap A_I(\C^k)\big).
	$$
	The set $(2\pi i\Z)^{|I|}\cap A_I(\C^k)$ is discrete in $A_I(\C^k)$, and the preimage of a discrete set under a linear isomorphism is discrete; hence $G_z^{\mathrm{eff}}$ is discrete.
	Finally, the projection $\C^k=K_I\oplus F_I\to F_I$ induces the claimed isomorphism 
	$$ 
		\C^k/G_z\cong F_I/G_z^{\mathrm{eff}}.
	$$
\end{proof}

\begin{theorem}[Threshold Theorem]\label{thm:threshold}
	Assume that the weight configuration is in the Poincaré domain and it satisfies the genericity condition.
	Let $\stratS{I}$ be a coordinate stratum and set $m=|I|$.
	Then the topology of the leaves in $\stratS{I}$ falls into the following three regimes.

	\begin{enumerate}

		\item If $m\le k$, then $\stratS{I}$ is a single leaf.
		      In particular, this leaf is biholomorphic to $(\C^*)^m$ and diffeomorphic to $\R^m\times\T^m$.

%		\item If $k<m<2k$, then the effective action on $\stratS{I}$ has complex
%		      dimension $k$, and
%		      $$
%				\dim_\R\bigl(A_I(\C^k)\cap i\R^m\bigr)\ge 2k-m>0.
%		      $$
%		      Moreover, for $z\in \stratS{I}$,
%		      $$
%			      G_z^{\mathrm{eff}}\neq\{0\}
%		      $$
%		      if and only if
%		      $$
%			      A_I(\C^k)\cap (2\pi i\Z)^m\neq\{0\}.
%		      $$
%		      Thus, in this intermediate range, the diffeomorphism type of the leaves, and even its topology,depends on the arithmetic position of $A_I(\C^k)$ with respect to the period lattice $(2\pi i\Z)^m$.

\item If $k<m<2k$, then the effective action on $\stratS{I}$ has
complex dimension $k$, and
\[
    \dim_{\R}\bigl(A_I(\C^k)\cap i\R^m\bigr)
    \geq 2k-m>0.
\]
Moreover, for every $z\in\stratS{I}$,
\[
    G_z^{\mathrm{eff}}\neq\{0\}
    \quad\Longleftrightarrow\quad
    A_I(\C^k)\cap(2\pi i\Z)^m\neq\{0\}.
\]
Thus, in this intermediate range, the diffeomorphism type of the leaves, and even their homeomorphism type, may depend on the arithmetic position of $A_I(\C^k)$ relative to the period lattice $(2\pi i\Z)^m$.

		\item If $m\ge 2k$, then, for an open dense set of weight configurations, one has $G_z^{\mathrm{eff}}=\{0\}$ for every $z\in \stratS{I}$.
		      Consequently, the leaves in $\stratS{I}$ are biholomorphic to $\C^k$ and diffeomorphic to $\R^{2k}$.
	\end{enumerate}
\end{theorem}

\begin{proof}
	By the genericity condition, the rank of the subset $\{\Lambda_j:j\in I\}$ is $q(I)=\min\{k,m\}.$
	Thus the effective action on $\stratS{I}$ has complex dimension $q(I)$.

	If $m\le k$, then $q(I)=m$.
	Hence the effective action has the same complex dimension as the stratum $\stratS{I}\simeq(\C^*)^m$.
	Since the exponential map 
	\[
		\C^m\longrightarrow(\C^*)^m
	\] 
	is surjective, the effective orbit through any point of $\stratS{I}$ is all of $\stratS{I}$.
	Thus $\stratS{I}$ is a single leaf, biholomorphic to $(\C^*)^m$.

	Assume now that $k<m<2k$.
	Then $q(I)=k$, and $A_I(\C^k)$ is a real $2k$-dimensional subspace of $\C^m\simeq\R^{2m}$.
	Since $i\R^m$ has real dimension $m$, the dimension formula gives 
	$$ 
		\dim_\R\bigl(A_I(\C^k)\cap i\R^m\bigr) \ge 2k+m-2m = 2k-m.
	$$
	This is positive precisely in the range $k<m<2k$.

	On the other hand, the effective isotropy is determined by the period condition 
	$$ 
		A_I(T)\in(2\pi i\Z)^m.
	$$
	Therefore
	$$
		G_z^{\mathrm{eff}}\neq\{0\}
	$$
	if and only if
	$$
		A_I(\C^k)\cap(2\pi i\Z)^m
	$$
	contains a nonzero lattice point.
	Hence the diffeomorphism type of the leaf depends on this arithmetic intersection.

	Finally, suppose that $m\ge 2k$.
	For an open dense set of weight configurations, the real $2k$-plane $A_I(\C^k)\subset\C^m$ is transverse to $i\R^m$.
	Hence 
	$$ 
		A_I(\C^k)\cap i\R^m=\{0\}.
	$$
	Since $(2\pi i\Z)^m\subset i\R^m$, it follows that
	$$
		A_I(\C^k)\cap(2\pi i\Z)^m=\{0\}.
	$$
	Therefore $G_z^{\mathrm{eff}}=\{0\}$, and the corresponding leaves are biholomorphic to $\C^k$.
\end{proof}

\begin{theorem}[Generic leaf-equivalence in the high-dimensional regime]
	\label{thm:high-regime-stability}
	Let $I\subset\{1,\dots,n\}$, let $m=|I|$, and assume that $m\geq2k$.
	Let
	\[
		\mathcal C_I^{\mathrm{rk}} = \left\{ A:\C^k\longrightarrow\C^m: \rank_{\C}A=k \right\}
	\]
	be the space of restricted weight configurations of complex rank $k$.
	For $A\in\mathcal C_I^{\mathrm{rk}}$, put
	\[
		V_A=A(\C^k)\subset\C^m.
	\]
	Then
	\[
		\mathcal U_I = \left\{ A\in\mathcal C_I^{\mathrm{rk}}: V_A\cap i\R^m=\{0\} \right\}
	\]
	is open and dense in $\mathcal C_I^{\mathrm{rk}}$.

	Moreover, if $A,A'\in\mathcal U_I$, then there is a smooth diffeomorphism
	\[
		h:(\C^*)^m\longrightarrow(\C^*)^m
	\]
	which maps every leaf of the orbit foliation associated with $A$ onto a leaf of the orbit foliation associated with $A'$.
	Consequently, for every $A\in\mathcal U_I$ there exists a neighborhood
	\[
		\mathcal V\subset\mathcal C_I^{\mathrm{rk}}
	\]
	of $A$ such that, for every $A'\in\mathcal V$, the two orbit foliations on $(\C^*)^m$ are smoothly leaf-equivalent.
\end{theorem}

\begin{proof}
	We divide the proof into three steps. First, we establish that $\mathcal U_I$ is open and dense. We then obtain a normal form for the foliation, which is used in the final step to construct the desired leaf-equivalence.

	\smallskip
	\noindent\emph{Step 1: Openness and density of $\mathcal U_I$.}
	For $A\in\mathcal C_I^{\mathrm{rk}}$, consider the real-linear map
	\[
		\operatorname{Re}\circ A:\C^k\longrightarrow\R^m,
	\]
	where $\C^k$ is regarded as a real vector space of dimension $2k$.
	We claim that
	\[
		V_A\cap i\R^m=\{0\}
	\]
	if and only if $\operatorname{Re}\circ A$ is injective.
	Indeed, if $A(T)\in i\R^m$, then $\operatorname{Re}(A(T))=0$.
	Thus injectivity of $\operatorname{Re}\circ A$ implies $T=0$, and hence $A(T)=0$.
	Conversely, if $\operatorname{Re}\circ A$ has a nonzero kernel, then there is a nonzero $T\in\C^k$ with $A(T)\in i\R^m$; since $A$ has complex rank $k$, it is injective, so $A(T)\neq0$.

	Consequently, $\mathcal U_I$ is the set of $A$ for which the real $m\times 2k$ matrix of $\operatorname{Re}\circ A$ has maximal rank $2k$.
	Because $m\geq2k$, this is an open condition.
	It is dense because it is the nonvanishing of at least one real $2k\times2k$ minor, and this condition is nonempty: for example, the complex-linear map
	\[
		A_0(t_1,\dots,t_k) =(t_1,\dots,t_k,it_1,\dots,it_k,0,\dots,0)
	\]
	has $\operatorname{Re}\circ A_0$ of real rank $2k$.

	\smallskip
	\noindent\emph{Step 2: A normal form for the foliation.}
%	Fix $A\in\mathcal U_I$.
%	Since the real-part projection restricts to an isomorphism
%	\[
%		\operatorname{Re}:V_A\longrightarrow P_A:=\operatorname{Re}(V_A) \subset\R^m,
%	\]
%	every vector of $V_A$ is uniquely of the form
%	\[
%		x+iB_Ax,\qquad x\in P_A,
%	\]
%	for a real-linear map $B_A:P_A\to\R^m$.
%	Since $V_A$ is a complex vector subspace, multiplication by $i$ preserves $V_A$; it follows that $B_A(P_A)\subset P_A$ and that
%	\[
%		B_A^2=-\operatorname{id}_{P_A}.
%	\]
Fix $A\in\mathcal U_I$. Since the real-part projection
\[
    \operatorname{Re}:V_A\longrightarrow
    P_A:=\operatorname{Re}(V_A)\subset\R^m
\]
is an isomorphism, $V_A$ is the graph of a unique real-linear map
$B_A:P_A\to\R^m$; namely,
\[
    V_A=\{x+iB_Ax:x\in P_A\}.
\]
Since $V_A$ is a complex vector subspace,
\[
    i(x+iB_Ax)=-B_Ax+ix\in V_A.
\]
It follows that $B_Ax\in P_A$ and $B_A^2x=-x$ for every $x\in P_A$.
Hence
\[
    B_A(P_A)\subset P_A,
    \qquad
    B_A^2=-\operatorname{id}_{P_A}.
\]

%	Use logarithmic polar coordinates on $(\C^*)^m\cong\R^m\times\T^m$,
%	\[
%		z\longmapsto\bigl((\log|z_j|)_{j=1}^m,(\arg z_j)_{j=1}^m\bigr).
%	\]
%	In these coordinates, the leaves of the orbit foliation associated with
%	$A$ are
%	\[
%		\mathcal L^A_{(r,\theta)} = \left\{(r+x,\theta+B_Ax):x\in P_A \right\},
%	\]
%	where the second component is understood modulo $2\pi\Z^m$.

Using logarithmic polar coordinates
\[
    (\C^*)^m\cong\R^m\times\T^m,
    \qquad
    z\longmapsto
    \bigl((\log|z_j|)_{j=1}^m,(\arg z_j)_{j=1}^m\bigr),
\]
the leaves of the orbit foliation associated with $A$ take the form
\[
    \mathcal L^A_{(r,\theta)}
    =
    \{(r+x,\theta+B_Ax):x\in P_A\},
\]
where the second component is taken modulo $2\pi\Z^m$.

	\smallskip
	\noindent\emph{Step 3: Construction of the leaf-equivalence.}
	Let $A,A'\in\mathcal U_I$, and write
	\[
		P=P_A,\qquad B=B_A,\qquad P'=P_{A'},\qquad B'=B_{A'}.
	\]
	Since $P$ and $P'$ are both $2k$-dimensional subspaces of $\R^m$ and
	$GL(m,\R)$ acts transitively on the Grassmannian $\mathrm{Gr}(2k,m)$,
	there exists a real-linear automorphism $L:\R^m\to\R^m$ with $L(P)=P'$.
	Define a real-linear map $D:P\to\R^m$ by
	\[
		D|_P=B'\circ L|_P-B,
	\]
	and extend $D$ arbitrarily to a real-linear map $D:\R^m\longrightarrow\R^m$.
	Now define
	\[
		H:\R^m\times\T^m\longrightarrow\R^m\times\T^m, \qquad H(r,\theta)=(Lr,\theta+Dr).
	\]
	This is a smooth diffeomorphism, with inverse
	\[
		H^{-1}(r',\theta') = \left(L^{-1}r',\theta'-DL^{-1}r'\right).
	\]

	For $x\in P$, one has
	\[
		\begin{aligned}
			H(r+x,\theta+Bx) & =(Lr+Lx,\theta+Dr+Bx+Dx) \\ &=(Lr+Lx,\theta+Dr+B'Lx).
		\end{aligned}
	\]
	Since $Lx\in P'$, this belongs to the $A'$-leaf through $H(r,\theta)$.
	Hence $H$ maps every $A$-leaf onto an $A'$-leaf.
	Transporting $H$ back through logarithmic polar coordinates yields the required diffeomorphism $h:(\C^*)^m\to(\C^*)^m$.

	Finally, since $\mathcal U_I$ is open, every $A\in\mathcal U_I$ has a neighborhood $\mathcal V\subset\mathcal U_I$, and the preceding construction applies to every $A'\in\mathcal V$.
\end{proof}

The preceding theorem gives the complementary rigidity statement in the high-dimensional regime: on the open dense set $\mathcal U_I$, the orbit foliation on $\stratS{I}\cong(\C^*)^m$ is locally constant up to smooth leaf-equivalence.

Theorem~\ref{thm:threshold} shows that the topology of the leaves is rigid in the two extreme regimes, but becomes arithmetic in the intermediate one.
Thus the threshold region 
\[
	k<|I|<2k
\] 
is precisely where the diffeomorphism type of the leaves may change under perturbations of the weights.

This is the mechanism that obstructs stratified stability.
Indeed, any notion of stability which preserves the orbit stratification together with the diffeomorphism types of the leaves must also preserve the arithmetic ranks of the groups $G_z^{\mathrm{eff}}$.
The preceding theorem shows that, in the threshold region, these ranks are controlled by lattice intersections and hence are not purely geometric invariants.

\begin{lemma}[Lattice planes]
	\label{lem:lattice-instability}
	Let $1\le d<m$ and let $\Lambda_0=(2\pi i\Z)^m\subset i\R^m$.
	In the Grassmannian of real $d$-planes in $i\R^m$, the set of planes meeting $\Lambda_0$ nontrivially is dense, and so is its complement.
\end{lemma}

\begin{proof}
	Identify $i\R^m$ with $\R^m$ and $\Lambda_0$ with $\Z^m$.
	Rational $d$-planes are dense and meet $\Z^m$ in lattices of rank $d$.
	Conversely, for every $0\ne v\in\Z^m$, the condition $v\in P$ defines a proper closed subset of the Grassmannian.
	The union of these subsets is countable and has empty interior; its complement is therefore dense.
\end{proof}

\begin{lemma}[Local realization of intersection planes]
	\label{lem:realization}
	Let $W=i\R^m\subset\C^m$, and let $V\subset\C^m$ be a complex $k$-plane for which
	\[
		P=V\cap W
	\]
	has the minimal possible dimension $d=2k-m>0$.
	Then every sufficiently small deformation $P'$ of $P$ through real $d$-planes in $W$ is of the form
	\[
		P'=V'\cap W
	\]
	for a sufficiently small deformation $V'$ of $V$ through complex $k$-planes.
\end{lemma}

\begin{proof}
	Choose a real linear automorphism $g$ of $W$, arbitrarily close to the identity, with $g(P)=P'$.
	Its complex-linear extension $g_{\C}:\C^m\to\C^m$ preserves $W$.
	Setting $V'=g_{\C}(V)$ gives
	\[
		V'\cap W=g_{\C}(V)\cap W=g(V\cap W)=g(P)=P'.
	\]
	Since $g$ can be chosen arbitrarily close to the identity, so can $V'$
	be chosen close to $V$.
\end{proof}

\begin{theorem}[Arithmetic instability in the threshold region]
	\label{thm:arithmetic-instability}
	Let $I\subset\{1,\dots,n\}$ satisfy $k<|I|=m<2k$, and assume that the restricted weights $\{\Lambda_j:j\in I\}$ have complex rank $k$.
	Then arbitrarily close to every such configuration there are configurations, still of restricted rank $k$, for which the effective isotropy in $\stratSS{I}$ is trivial, and configurations for which it is nontrivial.
	Consequently, the homeomorphism type of the leaves in $\stratSS{I}$ is not locally constant in the space of weight configurations.
\end{theorem}

\begin{proof}
	Let $V=A_I(\C^k)\subset\C^m$.
	A sufficiently small perturbation makes $V$ transverse to $W=i\R^m$, hence makes
	\[
		P=V\cap W
	\]
	have its minimal dimension $d=2k-m>0$.
	By Lemma~\ref{lem:realization}, nearby deformations of $P$ are realized by nearby complex $k$-planes $V'$.
	By Lemma~\ref{lem:lattice-instability}, we may choose such a nearby $P'$ either avoiding or meeting $(2\pi i\Z)^m$ nontrivially.
	Since every nearby complex $k$-plane is the image of an injective map $A'_I:\C^k\to\C^m$, it is realized by a nearby restricted weight configuration.

	For $z\in\stratSS{I}$, Lemma~\ref{lem:isotropy-splitting} identifies the effective isotropy with the inverse image of
	\[
		V'\cap(2\pi i\Z)^m.
	\]
	Thus it is trivial in the first case and nontrivial in the second.
	If its rank is $r$, the corresponding leaf is diffeomorphic to
	\[
		\R^{2k-r}\times\T^r,
	\]
	so these perturbations change the leaf topology.
\end{proof}

%\begin{definition}
%	Two diagonal holomorphic actions are said to be \emph{stratified topologically equivalent} if there exists a homeomorphism $h:\C^n\longrightarrow\C^n$ such that:
%	
%	\begin{enumerate}
%
%		\item $h$ maps each coordinate stratum onto the corresponding coordinate
%		      stratum;
%
%		\item $h$ maps every leaf of the first orbit foliation onto a leaf of
%		      the second;
%
%		\item for every coordinate stratum, the homeomorphism type of its leaves
%		      is preserved.
%
%	\end{enumerate}
%\end{definition}

\begin{definition}
Two diagonal holomorphic actions are said to be \emph{stratified topologically equivalent} if there exists a homeomorphism
\[
    h:\C^n\longrightarrow\C^n
\]
that preserves each coordinate stratum and maps every leaf of the first orbit foliation onto a leaf of the second. Moreover, on each coordinate stratum, the two orbit foliations have the same leaf homeomorphism types.
\end{definition}

\begin{corollary}[Failure of stratified stability]
	\label{cor:no-stratified-stability}
	Assume that the orbit stratification contains a coordinate stratum $\stratS{I}$ in the threshold region 
	$$ 
		k<|I|<2k.
	$$
	Then the orbit foliation is not stratified topologically stable.
	More precisely, arbitrarily small perturbations of the weight configuration may change the homeomorphism type of the leaves in $\stratS{I}$.
\end{corollary}

\begin{proof}
	By Theorem~\ref{thm:arithmetic-instability}, in the threshold region the rank of the discrete isotropy group $G_z^{\mathrm{eff}}$ is not locally constant under perturbations of the weights.
	Since the homeomorphism type of the leaves in $\stratS{I}$ is determined by this rank, arbitrarily small perturbations may change the homeomorphism type of the leaves in that stratum.

	Therefore no stratified topological equivalence, in the sense defined above, can persist under all sufficiently small perturbations of the weight configuration.
	Hence the orbit foliation is not stratified topologically stable.
\end{proof}

The loss of stability is therefore an intrinsic consequence of the threshold phenomenon.

\section{Local transverse geometry}\label{sec:transverse}

The previous sections studied the orbit foliation associated with a generic diagonal holomorphic action from the point of view of its leafwise geometry.
In this section we investigate its transverse geometry.
For regular foliations, the theory of transversally holomorphic structures is classical and of fundamental importance; see, for instance, \cite{Ha,GHS,GM1,GM2,GM3, GGmS, GN,LN,Bru,Bru-Gh}.
Unlike the regular foliations appearing in the classical theory, the orbit foliations we envisage here are singular: the leaves have different dimensions, and the leaf dimension jumps from one stratum to another.
An appropriate framework for studying such decompositions is the theory of Stefan--Sussmann singular foliations.

Recall that a Stefan--Sussmann singular foliation on a manifold $M$ is a partition of $M$ into connected immersed submanifolds, called leaves, which is locally generated by a family of smooth (or holomorphic) vector fields.
The orbit theorem of Stefan and Sussmann asserts that the leaves are precisely the maximal integral manifolds of the corresponding distribution.

In our setting, the orbit decomposition of the diagonal holomorphic action is a holomorphic Stefan--Sussmann singular foliation, whose leaves are the $\C^k$-orbits.
Our objective is to show that this foliation carries a canonical local transverse holomorphic structure.

\subsection{Stefan--Sussmann transversals}
The starting point for the transverse geometry is the notion of a local transversal to the orbit foliation.
Since the leaves have varying dimensions, the codimension of such a transversal depends on the point.

%Let $\mathcal F$ denote the orbit foliation associated with the diagonal holomorphic action of $\C^k$ on $\C^n$.
%The leaf through $p\in\C^n$ is the orbit $\leaf{p}$, and we write 
%$$ 
%	q(p)=\dim_{\C}\leaf{p}.
%$$
%If the leaf through $p$ has complex dimension $q(p)$, then a transversal at $p$ has complex dimension $n-q(p)$.

Let $\mathcal F$ denote the orbit foliation associated with the
diagonal holomorphic action of $\C^k$ on $\C^n$. The leaf through
$p\in\C^n$ is the orbit $\leaf{p}$, which has complex dimension
$q(p)$. Consequently, a transversal at $p$ has complex dimension
$n-q(p)$.

\begin{definition}
	A \emph{Stefan--Sussmann transversal} to $\leaf{}$ at $p$ is a germ of a complex submanifold $\Sigma_p\subset\C^n$ such that $p\in\Sigma_p$ and 
 	$$ 
		T_p\C^n = T_p\leaf{p} \oplus T_p\Sigma_p.
	$$
	In particular, $\dim_{\C}\Sigma_p=n-q(p)$.
\end{definition}

\begin{lemma}
	For every $p\in\C^n$ there exists a Stefan--Sussmann transversal $\Sigma_p$ to $\leaf{}$ at $p$.
\end{lemma}

\begin{proof}
	Since every orbit is a complex immersed submanifold of $\C^n$, the tangent space $T_p\leaf{p}$ is a complex vector subspace of $T_p\C^n$.
	Therefore it admits a complex linear complement $W_p$.
	Let $\Sigma_p$ be a complex submanifold through $p$ satisfying $T_p\Sigma_p=W_p$.
	Then
	$$
		T_p\C^n = T_p \leaf{p} \oplus T_p\Sigma_p,
	$$
	so $\Sigma_p$ is a Stefan--Sussmann transversal at $p$.
\end{proof}

The next result shows that transversality persists locally along the transversal.

\begin{lemma}
	Let $\Sigma_p$ be a  Stefan--Sussmann transversal at $p$.
	Then, after shrinking $\Sigma_p$ if necessary, 
	\[
		T_x\C^n = T_x\leaf{x} + T_x\Sigma_p
	\]
	for every $x\in\Sigma_p$.
\end{lemma}

%azer
\begin{proof}
	Let $q=q(p)$.
	Since $T_p\leaf{p}$ has complex dimension $q$, we may choose fundamental vector fields.
	$X_1,\ldots,X_q$ of the $\C^k$-action such that $X_1(p),\ldots,X_q(p)$ form a basis of $T_p\leaf{p}$.

	Since $\Sigma_p$ is transversal to $\leaf{p}$ at $p$, we have 
	$$ 
		T_p\C^n = \operatorname{span}_{\C}\{X_1(p),\ldots,X_q(p)\} \oplus T_p\Sigma_p.
	$$

	The vector fields $X_1,\ldots,X_q$ are holomorphic, and the tangent spaces $T_x\Sigma_p$ vary holomorphically with $x$.
	Hence the condition 
	\[
		T_x\C^n = \operatorname{span}_{\C}\{X_1(x),\ldots,X_q(x)\} + T_x\Sigma_p
	\]
 	is open.
	Therefore, after shrinking $\Sigma_p$ around $p$, it holds for every $x\in\Sigma_p$.

	Since each $X_j(x)$ is tangent to the orbit $\leaf{x}$, we have 
	$$ 
		\operatorname{span}_{\C}\{X_1(x),\ldots,X_q(x)\} \subset T_x\leaf{x}.
	$$
	Thus $ T_x\C^n = T_x\leaf{x} + T_x\Sigma_p $ for every $x\in\Sigma_p$.
\end{proof}

\subsection{Local product structure}

The local product structure constructed below provides convenient coordinates for the transverse geometry.
Its main role is to prove that the transverse orbit foliation is intrinsic, namely, independent of all auxiliary choices.

\begin{lemma}[Local product chart]\label{lem:Local product chart} 
	Let $\Sigma_p$ be a Stefan--Sussmann transversal at $p$, and let $q=q(p)$.
	Then, after shrinking $\Sigma_p$ if necessary, there exist a neighborhood $B$ of $0$ in $\C^q$, a neighborhood $V$ of $p$ in $\C^n$, and a biholomorphism 
	\[
		\Psi:B\times\Sigma_p\longrightarrow V
	\]
	such that 
	$\Psi(0,y)=y$ for every $y\in\Sigma_p$.
	Moreover, for every $y\in\Sigma_p$, the slice $\Psi(B\times\{y\})$ is contained in the ambient orbit $\leaf{y}$.
\end{lemma}

\begin{proof}
	Choose fundamental vector fields $X_1,\ldots,X_q$ such that $X_1(p),\ldots,X_q(p)$ form a basis of $T_p\leaf{p}$. 
	Since $\Sigma_p$ is a Stefan--Sussmann transversal at $p$, it follows that
	\[
    		T_p\C^n =  \operatorname{span}_{\C}\{X_1(p),\ldots,X_q(p)\} \oplus T_p\Sigma_p.
	\]

	For $t=(t_1,\ldots,t_q)$ sufficiently close to $0$ and $y\in\Sigma_p$ sufficiently close to $p$, define 
	$$ 
		\Psi(t,y) = \exp(t_1X_1)\circ\cdots\circ\exp(t_qX_q)(y).
	$$
%	Since the $\C^k$-action is abelian, the local flows of $X_1,\ldots,X_q$ commute, so the ordering of the factors is irrelevant and $\Psi$ is unambiguously defined.
%	The vector fields $X_1,\ldots,X_q$ are holomorphic fundamental vector fields of the action, so their local flows are holomorphic.
%	Therefore $\Psi$ is holomorphic.
%	Also, $\Psi(0,y)=y.$$
	Since the action is abelian, the fundamental vector fields $X_1,\ldots,X_q$ commute, and so do their local flows. Thus, $\Psi$ is independent of the order of composition. Moreover, the local flows depend holomorphically on both time and the initial point, and hence $\Psi$ is holomorphic. Finally, $\Psi(0,y)=y.$

%	Now we compute the differential of $\Psi$ at $(0,p)$.
%	The derivative in the $t_j$-direction is $X_j(p)$, while the derivative in the $\Sigma_p$-direction is the natural inclusion of $T_p\Sigma_p$ into $T_p\C^n$.
%	Thus 
%	\[
%		d\Psi_{(0,p)} : \C^q\oplus T_p\Sigma_p \longrightarrow T_p\C^n
%	\]
% 	is an isomorphism.
%
%	By the holomorphic inverse function theorem, after shrinking $B$ and $\Sigma_p$ if necessary, $\Psi$ is a biholomorphism onto an open neighborhood $V$ of $p$.
%
%	Finally, for fixed $y\in\Sigma_p$, the map 
%	\[
%		t\longmapsto \Psi(t,y)
%	\]
%	is obtained by composing local flows of fundamental vector fields of the $\C^k$-action.
%	Hence its image is contained in the orbit $\leaf{y}$.
 
	We now compute the differential of $\Psi$ at $(0,p)$. For $(\xi,v)\in\C^q\oplus T_p\Sigma_p$, we have
	\[
		d\Psi_{(0,p)}(\xi,v) = \sum_{j=1}^q \xi_jX_j(p)+v.
	\]
	Since
	\[
		T_p\C^n = \operatorname{span}_{\C}\{X_1(p),\ldots,X_q(p)\} \oplus T_p\Sigma_p,
	\]
	the differential
	\[
    		d\Psi_{(0,p)}: \C^q\oplus T_p\Sigma_p \longrightarrow T_p\C^n
	\]
	is an isomorphism. By the holomorphic inverse function theorem, after shrinking $B$ and $\Sigma_p$ if necessary, $\Psi$ restricts to a biholomorphism onto an open neighborhood $V$ of $p$.

	Finally, for each fixed $y\in\Sigma_p$, the map $t\longmapsto\Psi(t,y)$ is obtained by composing local flows of fundamental vector fields of the action. Its image is therefore contained in the orbit $\leaf{y}$. 
\end{proof}

We shall refer to $\Psi$ as a \emph{local product chart}.

\begin{proposition}[Uniqueness of the transverse foliation]\label{Uniqueness transverse foliation}
	Let $\Sigma_1$ and $\Sigma_2$ be two Stefan--Sussmann transversals at a point $p$.
	Then, after shrinking them if necessary, the transverse orbit foliations induced on $\Sigma_1$ and $\Sigma_2$ are biholomorphically equivalent as germs at $p$.
\end{proposition}

\begin{proof}
%	Let $\Psi:B\times\Sigma_1\longrightarrow V$ be a local product chart provided by Lemma~\ref{lem:Local product chart}. Using the local product chart $\Psi$, we identify $\Sigma_1$ with the zero section $\{0\}\times\Sigma_1$.
	Let $\Psi:B\times\Sigma_1\longrightarrow V$ be the local product chart provided by Lemma~\ref{lem:Local product chart}, and identify $\Sigma_1$ with the zero section $\{0\}\times\Sigma_1$.
	Let
	$$
		\widetilde{\Sigma}_2:=\Psi^{-1}(\Sigma_2)\subset B\times\Sigma_1.
	$$

	Since $d\Psi_{(0,p)}$ identifies the tangent space to the fibers of 
	\[
		\pi_2:B\times\Sigma_1\longrightarrow\Sigma_1
	\]
	with $T_p\leaf{p} $, the transversality of $\Sigma_2$ to $\leaf{p}$ implies that $\widetilde{\Sigma}_2$ is transverse at $(0,p)$ to the fibers of $\pi_2$.
	Equivalently, the restriction 
	\[
		\pi_2|_{\widetilde{\Sigma}_2}: \widetilde{\Sigma}_2 \longrightarrow \Sigma_1
	\]
	has invertible differential at $(0,p)$.

%	By the holomorphic inverse function theorem, after shrinking if necessary, $\pi_2|_{\widetilde{\Sigma}_2}$ is a biholomorphism onto a neighborhood of $p$ in $\Sigma_1$.
	After shrinking $\Sigma_1$ and $\widetilde{\Sigma}_2$ if necessary, the holomorphic inverse function theorem implies that $\pi_2|_{\widetilde{\Sigma}_2}$ is a biholomorphism onto a neighborhood of $p$ in $\Sigma_1$.
%	Therefore $\widetilde{\Sigma}_2$ is the graph of a unique holomorphic map 
%	$$ 
%		a:\Sigma_1\longrightarrow B, \qquad a(p)=0.
%	$$
%	Thus
%	$$
%		\widetilde{\Sigma}_2 = \{(a(y),y):y\in\Sigma_1\}.
%	$$
	Therefore, $\widetilde{\Sigma}_2$ is the graph of a unique holomorphic map $a:\Sigma_1\longrightarrow B$, with $a(p)=0$; namely,
	\[
		\widetilde{\Sigma}_2 = \{(a(y),y):y\in\Sigma_1\}.
	\]
	Define the vertical holomorphic vector field 
	\[
		Y = \sum_{j=1}^q a_j(y)\frac{\partial}{\partial t_j}
	\]
	on $B\times\Sigma_1$.
	Observe that $Y(0,p)=0$, so its flow fixes the point $(0,p)$.

	Since $Y$ has no component in the $\Sigma_1$-direction, the coordinate $y$ remains constant along every integral curve of $Y$.
	Consequently, the coefficients $a_j(y)$ remain constant along each integral curve.
%	Hence, if $(t_0,y)$ is the initial point of an integral curve, its solution is 
%	$$ 
%		t(s)=t_0+s\,a(y), \qquad y(s)=y.
%	$$
%	Hence, the integral curve through $(t_0,y)$ is given by $t(s)=t_0+s\,a(y)$ and $y(s)=y$.
%	Therefore the time-one map of the flow of $Y$ is
%	$$
%		(t,y)\longmapsto(t+a(y),y).
%	$$
%	In particular, it sends the zero section
%	$$
%		\{0\}\times\Sigma_1
%	$$
%	biholomorphically onto the graph of $a$, namely onto $\widetilde{\Sigma}_2$.
	Therefore, the time-one map of the flow of $Y$ is given by $(t,y)\mapsto(t+a(y),y)$, and its restriction to the zero section $\{0\}\times\Sigma_1$ is a biholomorphism onto the graph of $a$, which is precisely $\widetilde{\Sigma}_2$.

%	It remains to show that the vector field $\Psi_*Y$ is tangent to the ambient Stefan--Sussmann foliation.
%	Let 
%	\[
%		\Phi_j^s=\exp(sX_j)
%	\]
% 	denote the local flow of $X_j$.
%	Since the action is abelian, the vector fields $X_1,\ldots,X_q$ commute, and therefore so do their flows.

It remains to prove that $\Psi_*Y$ is tangent to the ambient
Stefan--Sussmann foliation. For $j=1,\ldots,q$, let
$\Phi_j^s=\exp(sX_j)$ denote the local flow of $X_j$. Since the action
is abelian, the fundamental vector fields $X_1,\ldots,X_q$ commute,
and hence so do their local flows wherever the corresponding
compositions are defined.

%	Fix $j$.
%	Using the commutativity of the flows, write 
%	$$ 
%		\Psi(t,y) = \Phi_j^{t_j} \circ \Phi_1^{t_1} \circ\cdots\circ \widehat{\Phi_j^{t_j}} \circ\cdots\circ \Phi_q^{t_q}(y).
%	$$
%	If
%	$$
%		z = \Phi_1^{t_1} \circ\cdots\circ \widehat{\Phi_j^{t_j}} \circ\cdots\circ \Phi_q^{t_q}(y),
%	$$
%	then
%	$$
%		\Psi(t,y)=\Phi_j^{t_j}(z).
%	$$
%	Differentiating with respect to $t_j$, we obtain
%	$$
%		\frac{\partial\Psi}{\partial t_j}(t,y) = \frac{d}{ds}\bigg|_{s=t_j}\Phi_j^s(z) = X_j(\Phi_j^{t_j}(z)) = X_j(\Psi(t,y)).
%	$$
%	Equivalently,
%	$$
%		\Psi_*\!\left(\frac{\partial}{\partial t_j}\right)=X_j.
%	$$
 
Fix $j\in\{1,\ldots,q\}$. By commutativity of the flows, we may write
\[
    \Psi(t,y)=\Phi_j^{t_j}(z_j),
    \qquad
    z_j:=
    \Phi_1^{t_1}\circ\cdots\circ
    \widehat{\Phi_j^{t_j}}\circ\cdots\circ
    \Phi_q^{t_q}(y),
\]
where $z_j$ is independent of $t_j$. Therefore,
\[
    \frac{\partial\Psi}{\partial t_j}(t,y)
    =
    \left.\frac{d}{ds}\right|_{s=t_j}\Phi_j^s(z_j)
    =
    X_j\bigl(\Phi_j^{t_j}(z_j)\bigr)
    =
    X_j\bigl(\Psi(t,y)\bigr).
\]
Equivalently,
\[
    \Psi_*\!\left(\frac{\partial}{\partial t_j}\right)=X_j.
\]

%	To simplify notation, we continue to denote the coefficient functions $a_j\circ\pi_2\circ\Psi^{-1}$ simply by $a_j$.
%	Thus 
%	\[
%		\Psi_*Y = \sum_{j=1}^q a_jX_j,
%	\]
% 	which is a holomorphic linear combination of fundamental vector fields.
%	Therefore $\Psi_*Y$ is tangent to the Stefan--Sussmann foliation.

To simplify notation, we continue to write $a_j$ for the coefficient
function $a_j\circ\pi_2\circ\Psi^{-1}$. It follows that
\[
    \Psi_*Y=\sum_{j=1}^q a_jX_j.
\]
Since the coefficients $a_j$ are holomorphic and the $X_j$ are
fundamental vector fields, $\Psi_*Y$ is tangent to the ambient
Stefan--Sussmann foliation.

%	Since every integral curve of $\Psi_*Y$ is contained in an ambient leaf, its time-one flow preserves every ambient leaf.
%	Consequently, its restriction defines a biholomorphism of germs 
%	$$ 
%		F:(\Sigma_1,p)\longrightarrow(\Sigma_2,p).
%	$$
%	Moreover, for every ambient leaf  $\leaf{}$ meeting the neighborhood, $F$ maps each connected component of
%	$$
%		\leaf{}\cap\Sigma_1
%	$$
%	onto a connected component of
%	$$
%		\leaf{}\cap\Sigma_2.
%	$$
%	Hence $F$ identifies the transverse orbit foliations on $\Sigma_1$ and $\Sigma_2$ as germs at $p$.
	
Since $\Psi_*Y$ is tangent to the ambient foliation, each of its
integral curves remains in a single ambient leaf. Hence its time-one
flow preserves every ambient leaf and restricts to a biholomorphism of
germs
\[
    F:(\Sigma_1,p)\longrightarrow(\Sigma_2,p).
\]

Moreover, for every ambient leaf $\leaf{}$ meeting the neighborhood,
$F$ maps each connected component of $\leaf{}\cap\Sigma_1$ onto a
connected component of $\leaf{}\cap\Sigma_2$. Therefore, $F$ identifies
the germs at $p$ of the transverse orbit foliations on $\Sigma_1$ and
$\Sigma_2$.	
\end{proof}

The following is immediate from the theorem above:

\begin{corollary}
	Up to germ biholomorphism, the transverse orbit foliation is independent of the choice of Stefan--Sussmann transversal.
\end{corollary}

\bigskip

\subsection{The transverse pseudogroup}
%The transverse orbit foliation admits a canonical description in terms of a local pseudogroup of biholomorphisms of a Stefan--Sussmann transversal, in the spirit of Haefliger's transverse geometry for regular foliations.
%
%Let $\Sigma_p$ be a Stefan--Sussmann transversal at $p$.
%Define the \emph{full transverse orbit pseudogroup} $\mathcal H_p$  to be the pseudogroup of germs of local biholomorphisms $h:U\to\Sigma_p$ such that, for every $x\in U$, the points $x$ and $h(x)$ belong to the same connected component of $\leaf{x}\cap\Sigma_p$.
%Equivalently, $\mathcal H_p$ consists of all local biholomorphisms of $\Sigma_p$ that preserve each leaf of the transverse orbit foliation setwise.
%This is the natural singular analogue of the holonomy pseudogroup of a regular foliation.

The transverse orbit foliation admits a canonical local description in
terms of a pseudogroup of biholomorphisms of a Stefan--Sussmann
transversal, in the spirit of Haefliger's transverse geometry for
regular foliations.

Let $\Sigma_p$ be a Stefan--Sussmann transversal at $p$. Define the
\emph{full transverse orbit pseudogroup} $\mathcal H_p$ as the
pseudogroup of all local biholomorphisms
\[
    h:U\longrightarrow V,
\]
where $U,V\subset\Sigma_p$ are open, such that, for every $x\in U$, the
points $x$ and $h(x)$ belong to the same connected component of
$\leaf{x}\cap\Sigma_p$. Equivalently, $\mathcal H_p$ consists of all
local biholomorphisms of $\Sigma_p$ that preserve the leaves of the
transverse orbit foliation. It is the natural singular analogue of the
holonomy pseudogroup of a regular foliation.

Let $\Psi:B\times\Sigma_p\longrightarrow V$ be a local product chart,
and define the associated holomorphic projection by
\[
    \rho:=\pi_2\circ\Psi^{-1}:V\longrightarrow\Sigma_p.
\]
 
For each sufficiently small $T\in\C^k$, let
$\Phi_T(z):=\Phi(T,z)$ denote the corresponding local action map.
Whenever $U\subset\Sigma_p$ is open and satisfies $\Phi_T(U)\subset V$,
define
\[
    h_T:=\rho\circ\left.\Phi_T\right|_U:
    U\longrightarrow\Sigma_p.
\]

%Since $\Phi_T$ maps each orbit $\leaf{x}$ into itself and preserves which connected component of $\leaf{x}\cap\Sigma_p$ a point belongs to (for $T$ sufficiently small), each $h_T$ belongs to $\mathcal H_p$.
%We call the maps $\{h_T\}$ the \emph{elementary action transformations}; they form a distinguished subpseudogroup of $\mathcal H_p$.

For $T$ sufficiently small, each point $x\in U$ and its image $h_T(x)$ belong to the same connected component of $\leaf{x}\cap\Sigma_p$.
Hence $h_T\in\mathcal H_p$.
We call the maps $h_T$ obtained in this way the \emph{elementary action transformations}. The pseudogroup they generate is a distinguished subpseudogroup of $\mathcal H_p$.

%\begin{proposition}
%	\textup{(i)}
%	The orbits of the subpseudogroup generated by the elementary action transformations $\{h_T\}$ are precisely the connected components of the transverse orbit foliation on $\Sigma_p$.
%
%	\textup{(ii)}
%	Consequently, the orbits of the full transverse orbit pseudogroup $\mathcal H_p$ are also precisely those connected components.
%\end{proposition}

\begin{proposition}
The orbits of the pseudogroup generated by the elementary action
transformations are precisely the leaves of the transverse orbit
foliation on $\Sigma_p$. Consequently, the full transverse orbit
pseudogroup $\mathcal H_p$ has the same orbits.
\end{proposition}

\begin{proof}
It suffices to prove the first assertion. Indeed, the pseudogroup generated by the elementary action transformations is contained in $\mathcal H_p$, so each of its orbits is contained in an $\mathcal H_p$-orbit. 
On the other hand, by definition, every element of $\mathcal H_p$ maps each point to another point in the same leaf of the transverse orbit foliation. 
Thus, once the first assertion is established, the orbits of $\mathcal H_p$ coincide with those leaves.

Let $\mathcal E_p$ denote the pseudogroup generated by the elementary
action transformations. We first show that the $\mathcal E_p$-orbit
of each $x\in\Sigma_p$ is precisely the connected component of
$\leaf{x}\cap\Sigma_p$ containing $x$.

Let $h_T$ be an elementary action transformation defined on an open
subset $U\subset\Sigma_p$, and let $x\in U$. After shrinking $U$ and
the parameter neighborhood if necessary, the path
\[
    [0,1]\longrightarrow\Sigma_p,
    \qquad
    t\longmapsto\rho\bigl(\Phi(tT,x)\bigr)
\]
is well defined. Since $\Phi(tT,x)\in\leaf{x}$ for every $t$ and the
fibers of $\rho$ are contained in ambient orbits, this path lies in
$\leaf{x}\cap\Sigma_p$. Its endpoints are $x$ and
$h_T(x)=\rho\bigl(\Phi(T,x)\bigr)$. Thus, every elementary action
transformation maps each point into the same connected component of
its transverse leaf. The same is therefore true for every element of
$\mathcal E_p$.

Conversely, let $y$ belong to the connected component of
$\leaf{x}\cap\Sigma_p$ containing $x$, and choose a path
\[
    \gamma:[0,1]\longrightarrow\leaf{x}\cap\Sigma_p
\]
such that $\gamma(0)=x$ and $\gamma(1)=y$.

For each $s\in[0,1]$, consider the orbit map
\[
    \theta_{\gamma(s)}:\C^k\longrightarrow\leaf{x},
    \qquad
    T\longmapsto\Phi(T,\gamma(s)).
\]
Its differential at $0$ has rank equal to the dimension of the orbit.
After restricting to a complex linear complement of the infinitesimal
isotropy at $\gamma(s)$, the holomorphic submersion theorem gives a
local holomorphic parametrization of the orbit near $\gamma(s)$.
Consequently, there exists a neighborhood $W_s$ of $\gamma(s)$ in
$\leaf{x}$ such that any two points of $W_s$ are related by an action
parameter $T$ sufficiently close to $0$.

By compactness of $\gamma([0,1])$, we may choose a subdivision
\[
    0=s_0<s_1<\cdots<s_N=1
\]
such that, for each $i$, the points $\gamma(s_i)$ and
$\gamma(s_{i+1})$ lie in one of these neighborhoods. Hence there
exists $T_i\in\C^k$, sufficiently close to $0$, such that
\[
    \Phi(T_i,\gamma(s_i))=\gamma(s_{i+1}).
\]
After shrinking its domain if necessary, the elementary action
transformation $h_{T_i}=\rho\circ\Phi_{T_i}$ is defined near
$\gamma(s_i)$. Since $\rho|_{\Sigma_p}=\mathrm{id}$, we obtain
\[
    h_{T_i}\bigl(\gamma(s_i)\bigr)
    =
    \gamma(s_{i+1}).
\]
It follows that $y$ is obtained from $x$ by a finite composition of
elementary action transformations. Therefore, the orbits of
$\mathcal E_p$ are precisely the leaves of the transverse orbit
foliation.

Finally, $\mathcal E_p\subset\mathcal H_p$, so every
$\mathcal E_p$-orbit is contained in an $\mathcal H_p$-orbit. On the
other hand, by definition, every element of $\mathcal H_p$ preserves
each leaf of the transverse orbit foliation. Hence the
$\mathcal H_p$-orbits are contained in those leaves and therefore
coincide with them.

\end{proof}

%\begin{corollary}
%	Up to conjugacy by a germ of biholomorphism, the full transverse orbit pseudogroup is independent of the choice of Stefan--Sussmann transversal.
%\end{corollary}

\begin{corollary}
The conjugacy class of the full transverse orbit pseudogroup is
independent of the choice of Stefan--Sussmann transversal.
\end{corollary}

\begin{proof}
Let $\Sigma_1$ and $\Sigma_2$ be Stefan--Sussmann transversals at $p$,
with full transverse orbit pseudogroups $\mathcal H_1$ and
$\mathcal H_2$, respectively. By
Proposition~\ref{Uniqueness transverse foliation}, after shrinking the
transversals if necessary, there exists a biholomorphism
\[
    F:(\Sigma_1,p)\longrightarrow(\Sigma_2,p)
\]
that identifies their transverse orbit foliations.

Let $h\in\mathcal H_1$ be a local biholomorphism, and consider its
conjugate $F\circ h\circ F^{-1}$. For every point $y$ in its domain,
write $y=F(x)$. Since $h\in\mathcal H_1$, the points $x$ and $h(x)$
belong to the same leaf of the transverse orbit foliation on
$\Sigma_1$. As $F$ identifies the two transverse foliations, the points
\[
    y=F(x)
    \qquad\text{and}\qquad
    (F\circ h\circ F^{-1})(y)=F(h(x))
\]
belong to the same leaf of the transverse orbit foliation on
$\Sigma_2$. Hence
$F\circ h\circ F^{-1}\in\mathcal H_2$, and therefore
\[
    F\mathcal H_1F^{-1}\subset\mathcal H_2.
\]

Applying the same argument to $F^{-1}$ yields the reverse inclusion.
Thus
\[
    F\mathcal H_1F^{-1}=\mathcal H_2.
\]
\end{proof}

%
%
%
%\begin{proof}
%	Let $\Sigma_1$ and $\Sigma_2$ be two Stefan--Sussmann transversals at $p$, with full transverse orbit pseudogroups $\mathcal H_1$ and $\mathcal H_2$.
%	By Proposition~\ref{Uniqueness transverse foliation}, after shrinking if necessary, there exists a biholomorphism of germs 
%	\[
%		F:(\Sigma_1,p)\longrightarrow(\Sigma_2,p)
%	\]
%	which identifies the transverse orbit foliations.
%
%	Let $h\in\mathcal H_1$ be a local element.
%	By definition of $\mathcal H_1$, the map $h$ satisfies: for every $x$ in its domain, $x$ and $h(x)$ belong to the same connected component of $\leaf{x}\cap\Sigma_1$.
%	Since $F$ maps connected components of $\leaf{x}\cap\Sigma_1$ bijectively onto connected components of $\leaf{F(x)} \cap\Sigma_2$, the conjugate 
%	\[
%		F\circ h\circ F^{-1}
%	\]
%	sends each point $y\in\Sigma_2$ in its domain to a point in the same connected component of $\leaf{y}\cap\Sigma_2$.
%	By definition of $\mathcal H_2$, this means $F\circ h\circ F^{-1}\in\mathcal H_2$.
%	Thus 
%	$$ 
%		F\mathcal H_1F^{-1}\subset \mathcal H_2.
%	$$
%	Applying the same argument to $F^{-1}$ gives the reverse inclusion.
%	Therefore 
%	$$ 
%		F\mathcal H_1F^{-1}=\mathcal H_2.
%	$$
%	Thus the full transverse orbit pseudogroup is independent of the choice of
%	transversal, up to conjugacy by a germ of biholomorphism.
%\end{proof}

\bigskip
The results of this section show that, under the maximal-rank hypotheses considered above, the orbit foliation associated with a diagonal holomorphic action carries an intrinsic local transverse geometry.
More precisely, although its construction involves the choice of a Stefan--Sussmann transversal and a local product chart, the resulting transverse foliation and its associated full transverse orbit pseudogroup are well defined up to germ biholomorphism.
Thus, to every point of the ambient singular foliation, one may naturally associate a germ of transversally holomorphic singular foliation, which provides a local transverse model for the orbit decomposition.

This intrinsic transverse geometry provides the natural framework for the study of local transverse invariants.
In particular, it opens the way to the construction of analytic transverse models and to the investigation of the local leaf space of the orbit foliation, topics that will be developed elsewhere.

\section*{Acknowledgments}

The authors acknowledge the use of large language models for proofreading, verifying calculations, and improving the clarity of the exposition. All mathematical ideas, arguments, and results are the authors’ own, and the authors take full responsibility for the content of the manuscript.

The authors also acknowledge financial support from PAPIIT Projects IN103324 and IN104826 (DGAPA, UNAM, Mexico).

\end{document}